\documentclass[10pt,twoside,a4paper]{amsart}
\usepackage[english]{babel}
\usepackage[utf8]{inputenc}

\usepackage{amsfonts,amsthm,amssymb,amsmath,amsthm,latexsym,wasysym,stmaryrd,mathrsfs,dsfont,txfonts}
\usepackage{accents,multirow,rotating,subfigure,graphicx,bigstrut,lscape,multicol,fancyhdr,enumerate,accents,mathtools,exscale}

\usepackage{pdftricks}
\usepackage[pdftex,all]{xy}

\usepackage{hyperref}
\usepackage{appendix}
\usepackage[justification=centering]{caption}

\graphicspath{{Figures/}}

\theoremstyle{theorem}
\newtheorem {theo}{Theorem}[section]
\newtheorem*{theo*}{Theorem}
\newtheorem {lemme}[theo]{Lemma}
\newtheorem*{lemme*}{Lemma}
\newtheorem {prop}[theo]{Proposition}
\newtheorem*{prop*}{Proposition}
\newtheorem {cor}[theo]{Corollary}
\newtheorem*{cor*}{Corollary}

\newtheorem*{cor_proof*}{Corollary (of the proof)}

\newtheorem*{conjecture*}{Conjecture}
\theoremstyle{definition}
\newtheorem {defi}[theo]{Definition}
\newtheorem*{defi*}{Definition}

\newtheorem*{nota*}{Notation}
\theoremstyle{remark}
\newtheorem {remarque}[theo]{Remark}
\newtheorem*{remarque*}{Remark}

\newtheorem*{warning*}{Warning}

\newtheorem*{remarques*}{Remarks}

\newtheorem*{warnings*}{Warnings}

\newtheorem*{convention*}{Convention}

\newtheorem*{exemple*}{Example}

\newtheorem*{exemples*}{Examples}
\newtheorem {fundex}[theo]{Fundamental Example}

\newtheorem*{question*}{Question}

\newtheorem*{questions*}{Questions}

\newtheorem*{fact*}{Fact}
\newtheorem*{acknowledgments}{Acknowledgments}

\def\DD{{\mathcal D}}
\def\N{{\mathds N}}
\def\R{{\mathds R}}

\def\Z{{\mathds Z}}

\def\p{\partial}

\def\F{\textrm{F}}
\def\Fn{\F_n}

\def\rPS{\textrm{r}\textrm{SL}}
\def\RF{\textrm{RF}}
\def\RFn{\RF_n}

\def\Tube{\textrm{Tube}}

\def\wSL{\textrm{w}\mathcal{SL}}

\def\tDD{\widetilde{\DD}}
\def\Cl{\textrm{Cl}}

\definecolor{purple}{rgb}{0.63, 0.36, 0.94}

\newcommand{\mvdir}[1]{(\textrm{#1})_\rightarrow}
\newcommand{\mvind}[1]{(\textrm{#1})_\leftarrow}
\newcommand{\mvdirind}[1]{(\textrm{#1})_\leftrightarrow}

\newcommand{\fract}[2]{\hbox{\leavevmode
  \kern.1em \raise .25ex \hbox{\the\scriptfont0 $#1$}\kern-.1em }\big/
  {\hbox{\kern-.15em \lower .5ex \hbox{\the\scriptfont0 $#2$}}}}

\begin{document}

\title{On codimension two embeddings up to link-homotopy}
\author[B. Audoux]{Benjamin Audoux}
         \address{Aix Marseille Universit\'e, I2M, UMR 7373, 13453 Marseille, France}
         \email{benjamin.audoux@univ-amu.fr}
\author[J.B. Meilhan]{Jean-Baptiste Meilhan}
         \address{Universit\'e Grenoble Alpes, IF, 38000 Grenoble, France}
         \email{jean-baptiste.meilhan@univ-grenoble-alpes.fr}
\author[E. Wagner]{Emmanuel Wagner}
         \address{IMB UMR5584, CNRS, Univ. Bourgogne Franche-Comt\'e, F-21000 Dijon, France}
         \email{emmanuel.wagner@u-bourgogne.fr}
\date{\today}

\begin{abstract}
We consider knotted annuli in $4$--space, called $2$--string links, which are knotted surfaces in codimension two that are naturally related, 
via closure operations, to both $2$--links and $2$--torus links. 
We classify $2$--string links up to link-homotopy by means of a $4$--dimensional version of Milnor invariants. 
The key to our proof is that any $2$--string link is link-homotopic to a ribbon one; this allows to use the homotopy classification obtained in the ribbon case by P.~Bellingeri and the authors. 
Along the way, we give a Roseman-type result for immersed surfaces in $4$--space. 
We also discuss the case of ribbon $k$--string links, for $k\geq 3$. 
\end{abstract}

\maketitle

\section{Introduction}
The study of knotted objects of several components up to link-homotopy was initiated by J.~Milnor in \cite{Milnor}. 
Roughly speaking, a \emph{link-homotopy} is a continuous deformation during which distinct components remain disjoint, but each
component may intersect itself. 
Studying knotted objects of several components up to link-homotopy is thus very natural, 
since it allows to unknot each component individually, 
and only records their mutual interactions; this is, in some sense, ``studying links modulo knot theory''. 

In the usual context of $1$--dimensional knotted objects in $3$--space, the first results were given by Milnor himself, 
who showed that his $\overline{\mu}$--invariants classify links with at most $3$ components up to link-homotopy. 
The case of $4$--component links was only completed thirty years later by J.~Levine, using a refinement of Milnor invariants \cite{Levine}.
A decisive step was then taken by N.~Habegger and X.~S.~Lin, who showed that Milnor invariants are actually well-defined invariants for \emph{string links}, 
{\it i.e.} pure tangles without closed components, and that they classify string links up to link-homotopy for any number of components \cite{HL}. 
\medskip

In the study of higher dimensional knotted objects in codimension $2$, the notion of link-homotopy 
seems to have first been studied by W.~S.~Massey and D.~Rolfsen for $2$--component $2$--links, {\it i.e.} two $2$--spheres embedded in $4$--space \cite{MR}.
In the late nineties, the study of $2$--links up to link-homotopy was definitively settled by A.~Bartels and P.~Teichner, who showed in \cite{BT} 
that all $2$--links are link-homotopically trivial. Actually, their result is much stronger, as it holds in \emph{any} dimension.
However, other classes of knotted surfaces in $4$--space remain quite interesting. 
In particular, in view of Habegger-Lin's work, it is natural to consider \emph{$2$--string links}, 
which are properly embedded annuli in the $4$--ball with prescribed boundary (see Definition \ref{def:sl}).\footnote{The terminology 
``$2$--string link'' is also sometimes used in the literature for a $2$--component ($1$--dimensional) string link ; 
since we are dealing here with surfaces in $4$--space, no confusion should occur.  }
One advantage of $2$--string links, as opposed to $2$--links, is that they carry a natural composition rule. 
Moreover, there are canonical closure operations turning a $2$--string link into a $2$--link or into a $2$--torus link, 
making this notion also relevant to the understanding of the more classically studied knotted spheres and tori in $4$--space. 
\medskip

The main result of this paper is the following. 
{
\renewcommand{\thetheo}{\ref{thm:2}}
\begin{theo}
Milnor $\mu^{(4)}$--invariants classify $2$--string links up to link-homotopy.
\end{theo}
\addtocounter{theo}{-1}
}
\noindent Here, the classifying invariant is a $4$--dimensional version $\mu^{(4)}$ of Milnor invariants, 
which is very natural in view of the classical $3$--dimensional case \cite{HL}. 
More precisely, we obtain that the group of $n$--component $2$--string links up to link-homotopy has rank 
%$\sum_{k=2}^n \binom{n}{k}k!$
$\sum_{k=2}^n \frac{n!}{(n-k)!(k-1)}$ (see Remark \ref{rem:milnor}).
This is in striking contrast with the case of $2$--links. 
\medskip

Theorem \ref{thm:2} relies on two main ingredients, 
which both involve the subclass of \emph{ribbon $2$--string links}, 
{\it i.e.} $2$--string links bounding immersed $3$--balls with only ribbon singularities 
(see Definition \ref{def:rsl}). 
The first ingredient is that, although seemingly very special, this subclass turns out to be 
generic up to link-homotopy: 
{
\renewcommand{\thetheo}{\ref{thm:1}}
\begin{theo}
Any $2$--string link is link-homotopic to a ribbon one. 
\end{theo}
\addtocounter{theo}{-1}
}

The second ingredient is a recent work of P.~Bellingeri and the authors \cite{ABMW}, 
which uses \emph{welded} knot theory to give a link-homotopy classification of ribbon $2$--string links\footnote{In \cite{ABMW}, ribbon $2$--string links are called ribbon tubes. }  
(see however Remark \ref{rem:ribbonlh} below).
Observing that the result of \cite{ABMW} reformulates in terms of Milnor $\mu^{(4)}$--invariants, 
and showing that these are invariant under link-homotopy, we thus obtain Theorem \ref{thm:2}.

The strategy of proof of Theorem \ref{thm:1} can be roughly outlined as follows. 
By shrinking and stretching a neighborhood of the boundary, a $2$--string link $T$ can be regarded as a $2$--link $L_T$ with thin, 
unknotted tubes attached, called \emph{outer annuli} below. 
Owing to \cite[Thm.~1]{BT}, there exists a link-homotopy from the $2$--link $L_T$ to the trivial $2$--link. 
Generically, this link-homotopy is link-homotopic to a composition of finger moves, Whitney tricks, 
cusp homotopies---each involving a single component---and isotopies.  
We are thus left with proving that these deformations can be always performed on the $2$--link $L_T$ 
so that they only produce ribbon-type linking with the outer annuli. 

For this purpose, we develop a diagrammatic theory for immersed surfaces in $4$--space. 
We introduce three singular Roseman moves, which are local moves on singular surface diagrams (see Figure \ref{fig:RosemanMoves}), 
and prove the following, which generalizes Roseman's theorem on embedded surfaces \cite{Roseman}. 
{
\renewcommand{\thetheo}{\ref{prop:singularRosemanMoves}}
\begin{prop}
Two singular surface diagrams represent (link-)homotopic immersed surfaces in $4$--space
if and only if they are connected by a finite sequence of Roseman moves and singular Roseman (self-)moves. 
\end{prop}
\addtocounter{theo}{-1}
}
This implies in particular a Roseman-type result for isotopies of immersed surfaces in $4$--space, 
which only involves one additional singular move (see Corollary \ref{cor:singrose}),  
and which generalizes a result given for $2$-braids by Kamada \cite{kamadaOJM}. 
\medskip 

Theorem \ref{thm:1} thus provides a generalization of the link-homotopy classification for ribbon $2$--string links of  \cite{ABMW}, 
and shows the relevance of the subclass of ribbon objects. 
The main result of \cite{ABMW} also raised the question of higher dimensional ribbon knotted objects in codimension two. 
The final section of this paper provides an answer by classifying ribbon $k$--string links up to link-homotopy, for all $k\ge 2$. 
This result builds on the homotopy classification of welded string links of \cite{ABMW}, 
combined with higher dimensional analogues of S.~Satoh's Tube map \cite{Satoh}. 
The key point here, which might be well-known to the experts, is that higher dimensional ribbon knotted objects in codimension $2$ 
essentially ``stabilize'' at dimension $2\looparrowright4$.

\begin{remarque}\label{rem:ribbonlh}
The notion of link-homotopy considered in \cite{ABMW} to classify ribbon $2$--string links may \emph{a priori} seems weaker than the usual notion, 
considered here. 
In \cite{ABMW}, we actually consider the equivalence relation generated by the self-circle crossing change operation, which locally replaces the over/under information 
at a circle of double points in a generic diagram, see \cite[\S~2.3]{ABMW}. Clearly, this operation can be realized by a regular link-homotopy. 
Conversely, it follows from Theorem \ref{thm:2} and \cite[Thm.~2.34]{ABMW} that two link-homotopic ribbon $2$--string links are necessarily related by a sequence of isotopies and self-circle crossing changes. In other words, the two notions coincide for ribbon $2$--string links. 
\end{remarque}

\begin{remarque}\label{rem:braidclosure}
Recall that a $2$--torus link is a smooth embedding of disjoint tori in $4$--space. 
Given a $2$--string link, there is a natural ``braid-like'' closure operation that yields a $2$--torus link.  
In \cite{ABMW}, the classification of ribbon $2$--string links up to link-homotopy was promoted to one for ribbon $2$--torus links, 
using the Habegger-Lin classification scheme of
\cite{HL2}. Unfortunately, the same method cannot be used in our more
general context since the above braid-closure map from 2--string links
to $2$--torus links is not surjective, even up to link-homotopy; see Appendix \ref{app:NonSurjectivity}.
It would be very interesting to achieve a general homotopy classification of $2$--torus links, 
and to compare it to our and Bartels--Teichner's results. 
\end{remarque}

\begin{remarque}
Throughout this paper, we will be working in the smooth category. 
We point out that the main result of \cite{ABMW} which we are using here is stated for locally flat objects; 
but since we are considering ribbon surfaces, there is no obstruction for approximating them by smooth objects.
\end{remarque}

The paper is organized as follows. 
We begin in Section \ref{sec:Ribbon} by reviewing $2$--string links and their ribbon version. In Section \ref{sec:roseman}, we introduce singular broken surface diagrams and singular Roseman moves, and we give a Roseman-type result for immersed surfaces in $4$--space.  
In Section \ref{sec:part1}, we prove Theorem \ref{thm:1} using singular broken surface diagrams. 
In Section \ref{sec:part2}, we review the definition of $4$--dimensional Milnor invariants and prove Theorem \ref{thm:2}. 
In the final Section \ref{sec:codimension2}, we give the link-homotopy classification of ribbon string links in higher dimensions. 

\begin{acknowledgments}
The authors would like to thank Peter Teichner and Akira Yasuhara for insightful discussions, 
and Louis Funar, whose question about the higher dimensional case led to the last section of this paper. 
We also thank the referee for his/her careful reading of the manuscript and for valuable comments. 
Thanks are also due to the GDR Tresses for providing support to start this project, and to the Isaac Newton Institute for Mathematical Sciences, Cambridge, 
for support and hospitality during the programme \emph{Homology theories in low dimensional topology}, 
where work on this paper was continued. E.W. wishes to thank 
the Universit\'e de Bourgogne for his CRCT, which facilitated this work.
\end{acknowledgments}

\section{Preliminaries}\label{sec:prelim}
In this section, we review the main objects of this paper---namely $2$--string links and their ribbon version---and the main tools used for their study---namely singular broken surface diagrams. 

\subsection{(Ribbon) $2$--string links}
\label{sec:Ribbon}

Fix $n\in\N^*$ disjoint Euclidian disks $D_1,\cdots ,D_n$ in the interior of the $3$--ball $B^3$. 
Denote by $C_i$ the oriented boundary of $D_i$. 

\begin{defi}\label{def:sl}
A ($n$--component) \emph{$2$--string link} is the isotopy class of a smooth embedding
 \[ \sqcup_{i=1}^n \left(S^1\times [0,1]\right)_i\hookrightarrow B^4 \]
of $n$ disjoint copies of the oriented annulus $S^1\times [0,1]$ into $B^4=B^3\times [0,1]$, 
such that the image of the $i$th annulus is cobounded by $C_i\times \{0\}$ and  $C_i\times \{1\}$, 
with consistent orientations.

Replacing ``embedding'' by  ``immersion with a finite number of double
points'', we obtain the notion of \emph{singular $2$--string link}.
\end{defi}

The natural operation of stacking product endows the set of $n$--component $2$--string links, denoted by $\textrm{SL}^{2}_n$, with a monoid structure, where the identity element is the trivial $2$--string link $\cup_i C_i\times [0,1]$.  

Given a $2$--string link $T$, the union of $T$ and the disks $D_i\times  \{\varepsilon\}$ for all $i=1,\cdots, n$ and $\varepsilon=0,1$ 
yields a $2$--link, {\it i.e.} a smooth embedding of $n$ copies of the $2$--sphere, in $4$--space. We call this $2$--link the \emph{disk-closure} of $T$.
There is another natural closure operation on $2$--string links where, as in the usual braid closure operation, one glues a copy of the trivial $2$--string link in the complement of $B^4$, thus producing a $2$--torus link. We shall call this operation the \emph{braid-closure map}; see Remark \ref{rem:braidclosure} and Appendix \ref{app:NonSurjectivity}.\\

As explained in the introduction, the following subclass of $2$--string links turns out to be quite relevant when working up to link-homotopy. 
\begin{defi}\label{def:rsl}
A $2$--string link $T$ is \emph{ribbon} if its disk-closure bounds $n$ immersed $3$--balls $B_1,\cdots,B_n$ such that
the singular set of $\cup_{i=1}^nB_i$ is a disjoint union of  \emph{ribbon singularities}, {\it i.e.} transverse disks whose
preimages are two disks, one lying in $\cup_{i=1}^n\mathring{B}_i$ and
the other having its boundary embedded in $T$.
\end{defi}

For any (singular) 2--string link $T$, we denote by $X(T)$ the complement of a tubular 
neighborhood of $T$ in $B^4$. We fix a basepoint which is far above $T$ (in a given direction that we shall use later to project
$T$ in $\R^3$), and define the \emph{fundamental group of $T$} as the fundamental group of $X(T)$ relative to this
basepoint.
We define now some special elements of the fundamental group.
For any point $p$ of $T$ which is regular (with respect to the chosen projection direction), we define the associated \emph{meridian}
as the loop which descends from the basepoint straight to $p$, turns positively around $p$ according to the combined orientations of $T$ and
$B^4$, and goes straight back to the basepoint. In particular, we define the \emph{$i$th bottom and top meridians} as the meridians
associated, respectively, to a point of $C_i\times\{0\}$ and $C_i\times\{1\}$.
Finally, we define an \emph{$i$th longitude for $T$} as an arc on the
boundary of a tubular neighborhood of the $i$th component of $T$, 
with fixed prescribed endpoints near  
$C_i\times\{0\}$ and $C_i\times\{1\}$, and closed by straight lines to the basepoint. It
can be noted that two $i$th longitudes differ by a number of $i$th bottom 
meridians; see \cite[\S~2.2.1]{ABMW} for more details. 

\subsection{Singular broken surface diagrams}\label{sec:roseman}

\emph{Broken surface diagrams} are the natural analogue of knot diagrams for
embedded surfaces in dimension 4. They correspond to generic projections of the
surfaces onto $\R^3$; this produces singularities, namely 1--dimensional loci of
double points and isolated triple and branch points. Double 
points are enhanced with an extra over/under information pictured by
erasing a small neighborhood of the undersheet. A finite set of local
moves, called \emph{Roseman moves} \cite{Roseman}, are known to generate the isotopy
relation, see Figure \ref{fig:RosemanMoves} for some examples. 
In this paper we shall use Roseman's original
notation $\textrm{(a)},\cdots,\textrm{(g)}$, as given in
\cite[Fig.1]{Roseman}, to denote them, possibly with an arrow
subscript if considering only a specific direction of the move. 
For example, $\mvdir{a}$ refers to Roseman move $\textrm{(a)}$ when applied from left to right in Figure \ref{fig:RosemanMoves}.

\begin{figure}[!h]
\[
\begin{array}{rc}
  \textrm{(a)}:&\vcenter{\hbox{\includegraphics{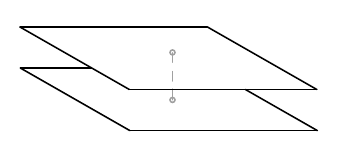}}}\
\longleftrightarrow\
                 \vcenter{\hbox{\includegraphics{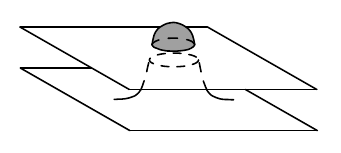}}}\\
  \textrm{(b)}:&\vcenter{\hbox{\includegraphics{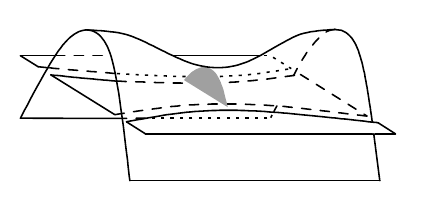}}}\
\longleftrightarrow\
                 \vcenter{\hbox{\includegraphics{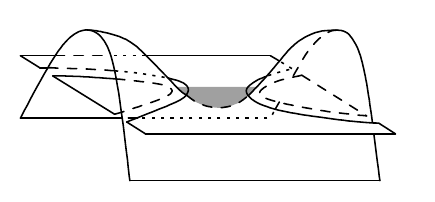}}}\\
  \textrm{(c)}:&\vcenter{\hbox{\includegraphics{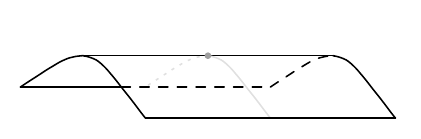}}}\
\longleftrightarrow\
                 \vcenter{\hbox{\includegraphics{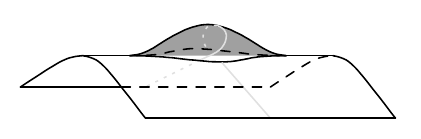}}}
\end{array}
\]
  \caption{The first three Roseman moves: {\scriptsize the move loci,
      in a neighborhood of which the moves are performed, are shown in
      dark grey; for move (c), a path on
  the surface has been drawn in light grey to help visualizing the picture}}
  \label{fig:RosemanMoves}
\end{figure}

In this section, we extend broken surface diagrams and Roseman moves to the singular setting. 
Generically, an immersed surface has a finite number of isolated singular double
points, and each of these singular points projects on an isolated point
inside a 1--dimensional locus of double points, where the over/under
information swaps; these singular double points shall be denoted by a
dot. See Figure \ref{fig:SingularRosemanMoves} for a few examples. 
A double point is called \emph{regular} if it is neither a triple, a branch, nor a singular point. 
\begin{defi}
A \emph{singular broken surface diagram} is a generic projection to $3$--space of an immersed surface in $4$--space, 
together with over/under information for each line or circle of regular double points. \\
The \emph{singular locus} of the diagram is the set of its double points, which contains in particular singular, branch and triple points.
\end{defi}
Of course, some additional moves on diagrams are required to generate isotopy and/or homotopy of immersed surfaces.  
These are the three \emph{singular Roseman moves} given, up to mirror image, in
Figure \ref{fig:SingularRosemanMoves}. Here, by \emph{mirror image}, we mean
the global swap of the over/under informations. 
A singular Roseman move shall be said to
be a \emph{self-move} if it involves singular points whose preimages belong to the same
connected component.

\begin{figure}[h!]
\[
\begin{array}{rc}
  \textrm{(h)}:&\vcenter{\hbox{\includegraphics{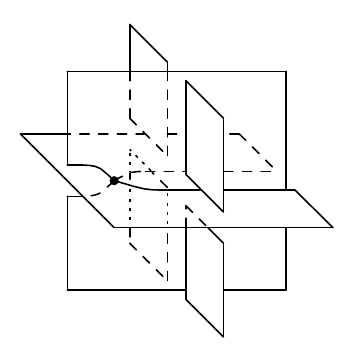}}}\
\longleftrightarrow\
                 \vcenter{\hbox{\includegraphics{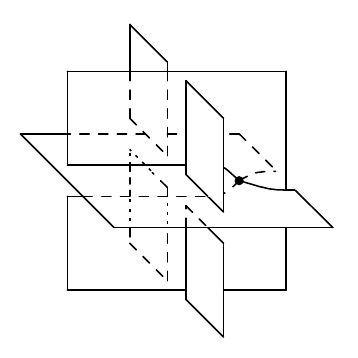}}}\\
  \textrm{(i)}:&\vcenter{\hbox{\includegraphics{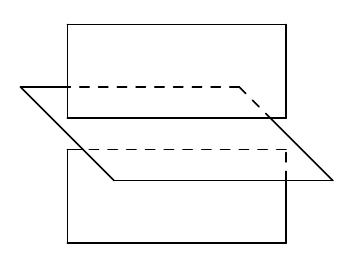}}}\
\longleftrightarrow\
                 \vcenter{\hbox{\includegraphics{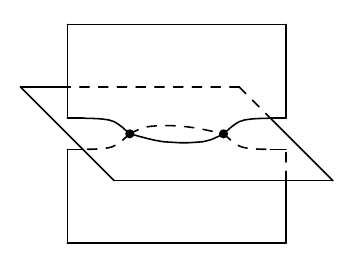}}}\\
  \textrm{(j)}:&\vcenter{\hbox{\includegraphics{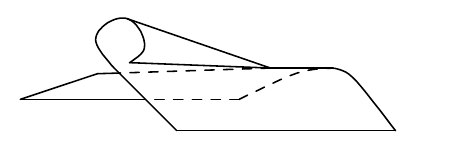}}}\
\longleftrightarrow\
                 \vcenter{\hbox{\includegraphics{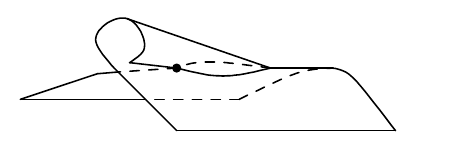}}}
\end{array}
\]
  \caption{Singular Roseman moves}
  \label{fig:SingularRosemanMoves}
\end{figure}

\begin{prop}\label{prop:singularRosemanMoves}
Two singular broken surface diagrams represent the same immersed surface in $4$--space up to
(link-)homotopy if and only if they are connected by a finite sequence of Roseman moves (a)--(g) and singular Roseman (self-)moves (h)--(j).
\end{prop}
\begin{proof}
We follow closely Roseman's approach in \cite{Roseman}, and we shall adopt his notation. 
Given a surface $M^2$ in $\R^4$, Roseman's proof amounts to understanding the singularities of a generic
homotopy $F:M^2\times [0,1]\rightarrow
\R^4 \times [0,1]$. The map $F$ is level preserving, 
{\it i.e.} for each $t\in [0,1]$ we have $F(M^2\times \{t\})\subset \R^4 \times \{t\}$. 
We consider the map $\pi\circ F:M^2\times [0,1]\rightarrow \R^3 \times [0,1]$ 
where the projection $\pi:\R^4 \times [0,1]\rightarrow \R^3\times [0,1]$ is the standard projection 
on the first factor and the identity on the second.
Following Roseman, we denote respectively by $B$, $D$, $T$ and $Q$ the
set of branch points, double points, triple points and quadruple
points, which are subsets of the interior of $M^2\times [0,1]$.
In addition, we define here $S$, the set of singular points. 
Recall that $B\subset D$ and that $Q\subset T\subset D$.
We similarly have $S\subset D$. Given $X$ in the interior of $M^2\times [0,1]$ 
we denote by $X^*$ its image through $\pi\circ F$. 
The strategy of Roseman is to consider the critical points of 
the composition of $\pi \circ F$ with the projection onto the last factor, restricted to $D^*$ 
(we can assume that $\pi \circ F$ is a Morse function restricted to $D^*$).
In addition to the analysis provided by Roseman, which takes care of $B^*$, $D^*$, $T^*$ and $Q^*$, we have to handle $S^*$, 
and there are four situations that we have to consider, recalling that $S^*$ is one dimensional inside $D^*$ which is two dimensional:
\begin{itemize}
\item $S^*$ intersects $T^*$: this corresponds to  move (h);
\item $S^*$ has a local maximum or minimum: this corresponds to move (i);
\item $S^*$ intersects $B^*$: this corresponds to move (j);
\item $S^*$ intersects itself: this corresponds to two singular points
  crossing one another along a line of double points and this is
  actually not generic. The pictures before and after the crossing are
  indeed the same, and the (link-)homotopy can be locally
  modified into a trivial one.\footnote{This corresponds to smoothing the
  crossing in one way, smoothing it in the other way would replace the
crossing of the two singular points by their mutual cancellation and re-creation.}
\end{itemize}

A link-homotopy is a special case of homotopy where all the singular
points involve twice the same connected component. 
In this case, any singular Roseman move arising in a sequence of moves (a)--(j) is necessarily a self-move.
\end{proof}

\begin{remarque}\label{rem:la_trick}
  It is well known, see {\it e.g.} \cite[p.20]{FQ} or \cite{Hirsch} that, generically,  
  (link-)homotopies are generated by finger, Whitney (self-)moves and cusp homotopies. 
  Proposition \ref{prop:singularRosemanMoves} actually provides
  a broken surface diagram proof of this statement in the smooth
  category. Indeed, move (i) can be seen as the broken surface diagram counterpart of
  finger/Whitney moves since, in their traditional representation, finger/Whitney moves are
  a combination of moves (a) and (i); and reciprocally, move (i) can
  be seen as a combination of finger/Whitney and (b) moves. Similarly,
  move (j) can be seen, up to (c) and (d) moves, as a broken surface diagram
  realization of the cusp homotopy.
\end{remarque}

An isotopy of singular immersed surfaces, with a finite number of
singular double points, can be
seen as a homotopy which preserves the singular points. The
proof of Proposition \ref{prop:singularRosemanMoves} hence implies the
following corollary which may be interesting on its own. 

\begin{cor}\label{cor:singrose}
Two singular broken surface diagrams represent isotopic immersed
surfaces in $4$--space if and only if they differ by a sequence of
Roseman moves (a)--(g) and of singular Roseman moves (h). 
\end{cor}

\begin{remarque}
In \cite{kamadaOJM}, Kamada proved, in term of charts, a
  similar statement for \emph{singular $2$--braids}.
  His result involves two extra moves, but the first one (move CIV) is actually used for the commutation of two faraway singular points in a $2$--braid, 
  and is thus not needed in our context. 
\end{remarque}

\section{Any $2$--string link is link-homotopically ribbon}\label{sec:part1}

In this section, we prove Theorem \ref{thm:1} stating that any $2$--string link is link-homotopic to a ribbon one. 
The proof uses the langage of singular broken surface diagrams and relies on Bartels--Teichner's theorem \cite[Thm.~1]{BT}.

\subsection{Pseudo-ribbon diagrams}

Singular $2$--string links have been defined as immersions of annuli $S^1\times [0,1]$ in the $4$--ball. 
The images of the circles $S^1\times \big\{\frac{1}{3}\big\}$ and $S^1\times \big\{\frac{2}{3}\big\}$ split each annulus into an \emph{inner annulus} and two \emph{outer annuli}. 

\begin{defi}\label{def:pseudo}
A \emph{pseudo-ribbon diagram} for a singular $2$--string link is a diagram such that 
the images of the circles $S^1\times \{\frac{1}{3}\}$ and $S^1\times \{\frac{2}{3}\}$ bound embedded $2$--disks, called \emph{attaching disks}, 
and such that  
\begin{itemize}
 \item the interior of the attaching disks are disjoint from the diagram;   
 \item the outer annuli meet the singular locus only at essential circles of regular double points, such that each of these essential circles bounds a disk in an inner annulus, whose interior is disjoint from the singular locus of the diagram. 
\end{itemize}
\end{defi}
Pseudo-ribbon diagrams should be thought of as diagrams of knotted spheres (the inner annuli), each with a pair of thin tubes attached (the outer annuli) which are the thickening of 1--dimensional cores, possibly linked with the spheres. 
In the figures, outer annuli shall be pictured with thick lines, and attaching disks will be shaded.
\begin{remarque}
The first condition in Definition \ref{def:pseudo} implies in particular that the boundary of the attaching disks are necessarily disjoint from the singular locus. 
The second condition implies that outer annuli are pairwise disjoint. 
\end{remarque}
In what follows, we will call \emph{outer circles} the essential
curves of regular double points on outer annuli, and \emph{outer
disks} the disks on the inner annuli which are bounded by outer circles.
Two pseudo-ribbon diagrams are called \emph{equivalent} if they
represent isotopic singular $2$--string links and
\emph{link-equivalent} if they represent link-homotopic singular $2$--string links.

We now introduce four local moves on pseudo-ribbon diagrams, shown below up to mirror image:
\begin{itemize}
\item move A passes an attaching disk across a line of regular
  double points between inner annuli;
\[
 \vcenter{\hbox{\includegraphics{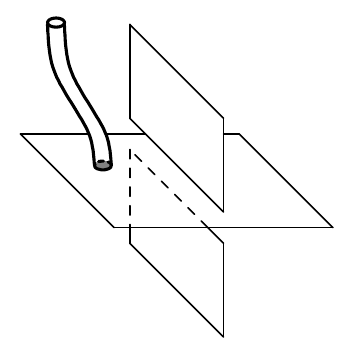}}}\ \stackrel{\textrm{A}}{\longleftrightarrow}\ \vcenter{\hbox{\includegraphics{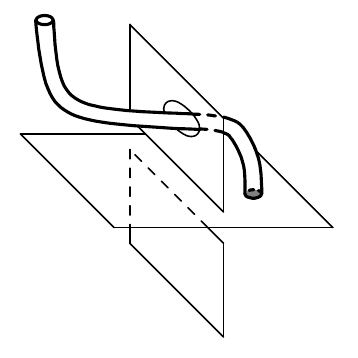}}}
\]
\item moves $\textrm{B}_1$ and $\textrm{B}_2$ pass an outer annulus across a line of regular
  double points between inner annuli;
  \[
\hspace{-.8cm}
 \vcenter{\hbox{\includegraphics{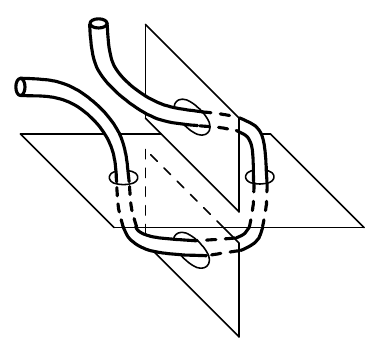}}}\
 \stackrel{\textrm{B}_1}{\longleftrightarrow}\
\vcenter{\hbox{\includegraphics{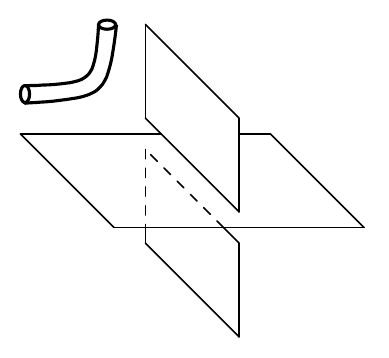}}}\ \stackrel{\textrm{B}_2}{\longleftrightarrow}\ \vcenter{\hbox{\includegraphics{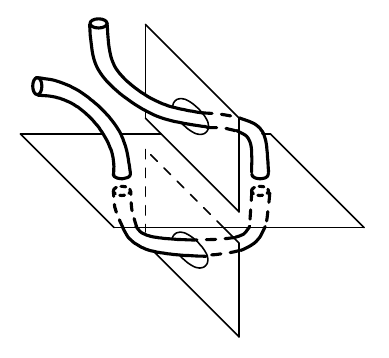}}}
\]

\item move C removes two nearby outer circles with same over/under information;
\[
\vcenter{\hbox{\includegraphics{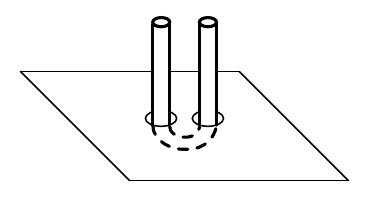}}}\ \stackrel{\textrm{C}}{\longleftrightarrow}\ \vcenter{\hbox{\includegraphics{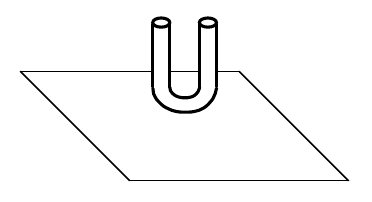}}}
\]
\item move D exchanges the relative 3--dimensional
  position of two outer annuli;
\[
\vcenter{\hbox{\includegraphics{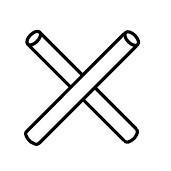}}}\ \stackrel{\textrm{D}}{\longleftrightarrow}\ \vcenter{\hbox{\includegraphics{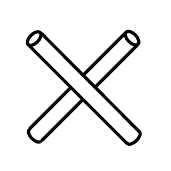}}}
\]
\end{itemize}
It should be noted that none of these moves involves any singular point.

The following is easily shown using Roseman moves: 
\begin{lemme}\label{lem:PseudoRibbonMoves}
 Two pseudo-ribbon diagrams which differ by a sequence of
  \begin{itemize}
\item moves A, $\textrm{B}_1$, $\textrm{B}_2$, C or D;
  \item Roseman or singular Roseman (self-)moves in a 3-ball which do 
not intersect any outer annulus;

  \end{itemize}
are (link-)equivalent.
\end{lemme}

The ``pseudo-ribbon'' terminology is justified by the following: 
\begin{lemme}\label{lem:PseudoRibbon}
A $2$--string link having a pseudo-ribbon diagram whose singular locus
consists only of outer circles, is ribbon.
\end{lemme}
\begin{proof}
This can be seen as a consequence of \cite[Lem.~2.12]{ABMW}, but
it can also easily be shown directly as follows. 
Since the inner annuli,
when closed by the attaching disks, have no singularity, they bound
embedded 3--balls $B_1,\cdots,B_n\subset B^3$. 
Inside each $B_i$, the outer annuli
can be pushed close to $\p B_i$ thanks
to move D, and possibly pushed out using move C. We are then left with a finite number of
\emph{hooks}, as pictured in Figure \ref{fig:Hook}.
Such a diagram can be easily lifted to a ribbon surface in $B^4$. 
\end{proof}
\begin{figure}[h!]
  \[
  \includegraphics{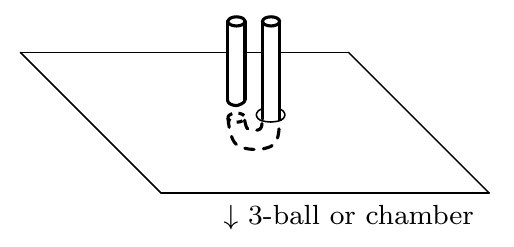}
  \]
  \caption{A hook between an outer annulus and either a 3--ball or a chamber}
  \label{fig:Hook}
\end{figure}

\subsection{Genericity of ribbon objects up to link-homotopy}

Ribbon objects are a very simple class of knotted surfaces. For genus
0 surfaces, they are known \cite[Thm.1]{Yaji} to correspond with
so-called \emph{simple} surfaces, which are surfaces admitting a projection with
only double points, and no triple nor branch point. For higher genus
surfaces, the ribbon class is even more restrictive, see {\it e.g.}
\cite[Ex.2.7]{CKS}. 
However, the next result shows that ribbon surfaces arise naturally when working up to link-homotopy.
\begin{theo}\label{thm:1}
Any $2$--string link is link-homotopic to a ribbon one. 
\end{theo}

\begin{proof}
  Let $T$ be a $2$--string link described by a broken surface diagram
  $\tDD$. By conjugating $\tDD$ with two trivial broken surface
  diagrams, we obtain a new broken surface diagram $\DD$ for $T$
  endowed with a pseudo-ribbon structure, for which $\tDD$ corresponds to the inner
  annuli and
  $\sqcup_{i=1}^n\big(D_i\times\big\{\frac13\big\}\sqcup
  D_i\times\big\{\frac23\big\}\big)$
  are the attaching disks. Note in particular that (a
  smoothing of) the union of the inner annuli and the attaching disks
  of $\DD$ is a broken surface diagram for the disk-closure of $T$, 
  that we shall denote by $\DD_\Cl$.

  By Bartels--Teichner's theorem \cite[Thm.1]{BT}, the disk-closure of $T$ is known to be
  link-homotopic to the trivial  2--link. Using Theorem
  \ref{prop:singularRosemanMoves}, it follows that there exists a
  sequence of Roseman moves and singular Roseman self-moves which transforms $\DD_\Cl$
  into the trivial broken surface diagram. Our goal is now to show
  that this sequence can be performed on the inner annuli of $\DD$
  while preserving its pseudo-ribbon structure. By use of Lemma
  \ref{lem:PseudoRibbon}, it will follow that the resulting 2--string
  link is ribbon.

  Obstructions for realizing the above-mentioned sequence may arise
  only when outer annuli interact with the 3--ball supporting one of the moves. 
  More precisely, Roseman moves and singular Roseman moves occur in a 3--ball, 
  which is a neighborhood of their ``locus'' 
  (see Figure \ref{fig:RosemanMoves} for examples), 
  and they can be classified in 4 types, depending on the dimension of this locus: 
  \begin{description}
  \item[Dimension 0] for instance, move $\mvdir{c}$ occurs in a
    neighborhood of the point where the two branch points shall
    appear.  Up to isotopy, this point can be chosen outside the
    attaching and outer disks. Moves $\mvdir{i}$ and $\mvdir{j}$ are
    part of the same class, but are even easier since the considered
    points are on the double point locus of the inner annuli, so they
    can't be contained in an attaching or outer disk.
  \item[Dimension 1] for instance, move $\mvdir{a}$ occurs in a
    neighborhood of a path which joins the two points, one on each sheet,
    that will merge to produce the circle of double points. Up to
    isotopy, the endpoints of this path can be chosen outside the
    attaching and outer disks, and the path can be chosen outside the outer
    annuli. Moves $\mvdir{e}$ and $\mvdir{f}$ can be handled
    similarly, and moves $\mvdir{d}$, $\mvdirind{h}$, $\mvind{i}$ and
    $\mvind{j}$ are part of the same class but are even easier, since
    the considered paths are on the double point locus of inner
    annuli, so they can't interact with any attaching or outer disk.
  \item[Dimension 2] for instance, move $\mvdirind{b}$ occurs in a
    neighborhood of a disk $D$ along which one of the sheets will be
    pushed; the interior of $D$ is disjoint from the diagram and its
    boundary is the union of two segments, one on each sheet. Up to
    isotopy, $\p D$ can be chosen outside the attaching and outer
    disks, and then the outer annuli which intersect $D$ can be pushed
    away using move C, as illustrated in Figure \ref{fig:Dim2Locus}. 
    Move $\mvind{d}$ can be handled similarly.
    \begin{figure}[h!]
      \[
      \vcenter{\hbox{\includegraphics{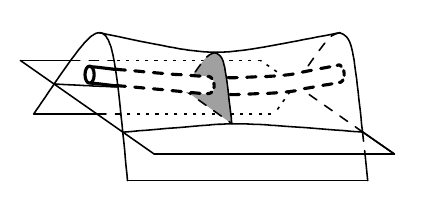}}} \ \leadsto\
      \vcenter{\hbox{\includegraphics{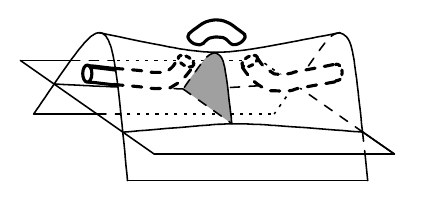}}}
      \]
      \caption{Avoiding a dimension two locus in a Roseman move $\mvdirind{b}$}
      \label{fig:Dim2Locus}
    \end{figure}
  \item[Dimension 3] for instance, move $\mvind{a}$ occurs in a
    neighborhood of a 3--dimensional \emph{chamber}, bounded by two pieces of
    sheet that we shall call \emph{walls}. Any attaching disk lying in
    one of these walls may be pushed away using move A. If an outer
    annulus enters and leaves the chamber through distinct walls, then
    moves $\textrm{B}_1$ and $\textrm{B}_2$ can be used to cut it in several pieces, each entering and
    leaving the chamber through the same wall; then thanks to move D, the
    outer annuli inside the chamber can be pushed close to the walls and
    possibly pushed out using move C. We are then left with
    hooks as in Figure \ref{fig:Hook}. Such hooks can
    be pushed out of the chamber using a combination of moves $\textrm{B}_1$, $\textrm{B}_2$ and C,
    as illustrated in Figure \ref{fig:Dehooking}. Moves $\mvind{c}$, $\mvind{f}$ and
    $\mvdirind{g}$ can be handled similarly. Move $\mvind{e}$ is also
    similar, but with three chambers, so one has to take care of
    emptying them successively in the right order.
    \begin{figure}
      \[
      \xymatrix@!0@R=2cm@C=5.5cm {
        \vcenter{\hbox{\includegraphics{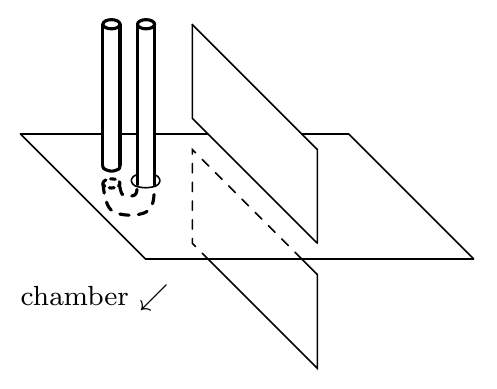}}}
        \ar@{}[r]^{\textrm{B}_1}|{\longrightarrow} &
        \vcenter{\hbox{\includegraphics{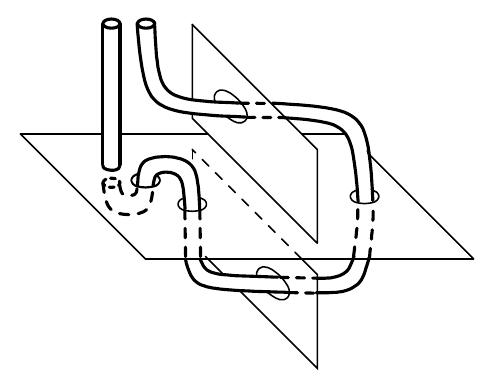}}}
        \ar@{}[rd]^{\textrm{C}}|{\rotatebox{330}{$\longrightarrow$}}&\\
        &&\vcenter{\hbox{\includegraphics{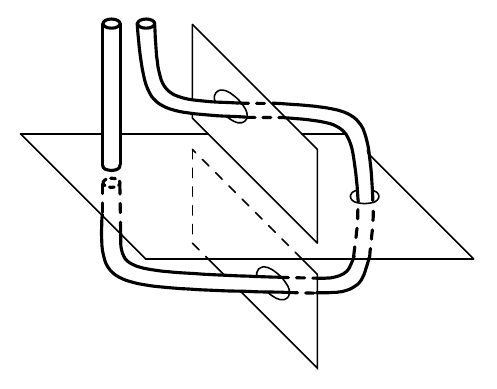}}}\ar@{}[ld]^{\textrm{B}_2}|{\rotatebox{30}{$\longleftarrow$}}\\
        \vcenter{\hbox{\includegraphics{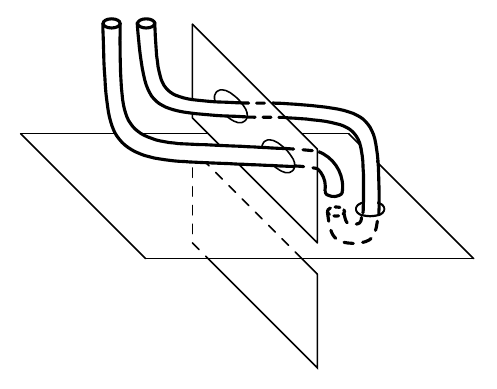}}} &
        \vcenter{\hbox{\includegraphics{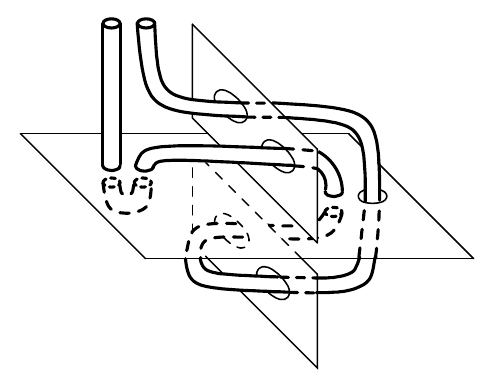}}}
        \ar@{}[l]^{\textrm{C}}|{\longleftarrow}& }
      \]
      \caption{Pushing a hook out of a chamber}
      \label{fig:Dehooking}
    \end{figure}
  \end{description}
  As a result of this discussion, up to moves A, $\textrm{B}_1$, $\textrm{B}_2$, C and D, every Roseman move and singular Roseman move can be
  performed away from the outer annuli. It  follows then from Lemma \ref{lem:PseudoRibbonMoves} that the
  pseudo-ribbon structure can be preserved all along the sequence.
\end{proof}

\section{Classification of $2$--string links up to link-homotopy}
\label{sec:part2}

Given any group $G$ with a fixed normal set of generators, we define the \emph{reduced group of $G$}, denoted by $\textrm{R}G$, 
as the smallest quotient where each generator commutes with all its
conjugates.
As we shall see, the reduced fundamental group of any $n$--component
2--string link $T$
is isomorphic to $\RFn$, where $\Fn$ is the free group generated by either the top or 
the bottom meridians of $T$.

We show in this section that, up to link-homotopy, 2--string links are actually classified by the
data of their longitudes in $\RFn$. In the literature, this invariant
appears in two different forms, either as an action on $\RFn$ or as $4$-dimensional Milnor invariants. 
We first review these two approaches, see \cite[\S~2.2.1.3 and \S~6.1]{ABMW} for more details.

\subsection{Conjugating automorphisms and Milnor invariants}
\label{sec:lhinvariants}
Let $T$ be a 2--string link with $n$ components $T_1,\cdots,T_n$, and $X(T)$ be the complement of a tubular
neighborhood of $T$. There are natural inclusions
$\iota_0,\iota_1: B^3\setminus \{ C_1 ,\cdots ,C_n\} \hookrightarrow X(T)$
which come, respectively, from the embedding of the
bottom and the top boundaries inside $B^3\times I$. It is straightforwardly checked that they induce isomorphisms at both the $H_1$ and $H_2$ levels, 
which by a theorem of Stallings  \cite[Thm.~5.1]{stallings} implies that they induce isomorphisms at the level of each nilpotent quotients\footnote{Recall that the nilpotent quotients of a group $G$ are defined by $\frac{G}{\Gamma_k G}$, where $\{\Gamma_k G\}_k$ is the lower central series of $G$.} 
of the fundamental groups. 
Now, the fundamental group of $B^3\setminus \{ C_1 ,\cdots ,C_n\}$
identifies with the free group $\Fn$ generated by the meridians $m_1,\cdots,m_n$
and, by \cite[Lem.~1.3]{HL}, the $k$th nilpotent quotient of $\Fn$ is equal to $\RFn$ for all $k\ge n$.
As a consequence, $\iota_0$ and $\iota_1$ induce isomorphisms
$\iota_0^*$ and $\iota_1^*$ at the level of the reduced fundamental group. 
By taking the composition ${\iota^*_0}^{-1}\circ\iota^*_1$, we can thus associate an automorphism of $\RFn$ to the $2$--string link $T$. 
It is easily seen that this assignment defines a monoid homomorphism 
$$ \varphi: \textrm{SL}^{2}_n \rightarrow \rm{Aut}_C(\RFn),$$
where $\rm{Aut}_C(RF_n)$ is the subgroup of \emph{conjugating automorphisms} of $\RFn$, mapping each generator to a conjugate of itself. 
More precisely, $\varphi(T)$ maps the $i$th generator to its conjugate
by (the image by $\iota^*_0$ of) \emph{any} $i$th longitude $\lambda_i$ for $T_i$. 

We now recall the definition of the non-repeated
$\mu^{(4)}$--invariants. 
For each $i\in\{1,\cdots,n\}$, we denote by
$\mathcal{S}^{(i)}_{n-1}:=\mathbb{Z}\langle\langle
X_1,\cdots,\hat{X}_i,\cdots,X_n\rangle\rangle$ the ring of formal
power series in $(n-1)$ non-commutative variables, 
and by $\F_{n-1}^{(i)}\cong\F_{n-1}$ the subgroup of $\Fn$
generated by all but the $i$th generator of $\Fn$. We also denote by $E_i:\F_{n-1}^{(i)}\rightarrow \mathcal{S}_{n-1}^{(i)}$ the Magnus
expansion, which is the group homomorphism
sending the $j$th generator to $1 + X_j$. 
This map descends to a well
defined homomorphism $E_i^h$ from $\RF^{(i)}_{n-1}$ to
$\fract{\mathcal{S}^{(i)}_{n-1}}{I_r}$, where $I_r$ is the ideal generated by monomials with repetitions. 
The \emph{$4$--dimensional Milnor invariant  $\mu^{(4)}_I(T)$ of $T$} is defined, 
for each sequence $I=i_1 \cdots i_k i$ of pairwise distinct integers in $\{1,\cdots,n\}$, 
 as the coefficient of the monomial $X_{i_1}\cdots 
 X_{i_k}$ in $E_i^h(\widetilde{\lambda}_i)$, where $\widetilde{\lambda}_i\in\RF^{(i)}_{n-1}$ is a 
 longitude for $T_i$ seen in the complement of $T\setminus T_i$. This is well
 defined since all longitudes for $T_i$ differ by some power of $m_i$, and are hence isotopic in the complement
 of $T\setminus T_i$. In particular, $\widetilde{\lambda}_i$ seen in $\RFn$ is actually an $i$th longitude, and it can reciprocally be obtained from any $i$th longitude by removing all $m_i$--factors.

\begin{lemme}\label{lem:equiv_inv}
 Two $2$--string links $T_1$ and $T_2$ have same Milnor $\mu^{(4)}$-invariants with non repeating indices if and only if $\varphi(T_1)=\varphi(T_2)$. 
\end{lemme}
\begin{proof}
 It is well known that $E_i^h$ is actually injective, see for example \cite[Thm.~7.11]{ipipipyura} for a proof. 
 From Milnor invariants, one can hence recover
 longitudes seen in $\RF_{n-1}^{(i)}\subset\RFn$ and hence the associated conjugating
 automorphism of $\RFn$.
 Reciprocally, it follows from \cite[Lem.4.25]{ABMW} that, for each $i\in\{1,\cdots,n\}$, one can
 extract $\widetilde{\lambda}_i\in\RF^{(i)}_{n-1}$ from the conjugating automorphism and hence recover Milnor invariants.
\end{proof}

These invariants are actually invariant under link-homotopy and this can be proven in several ways.
One can show directly the homotopy invariance of Milnor invariants, as in 
\cite{Milnor2}, using the effect of a finger or a cusp self-move at the level of $\pi_1$, see \emph{e.g.} \cite{Kirby}. 
Another approach shows the invariance of $\varphi$, in the spirit of \cite{HL,ABMW}, by considering the complement of a link-homotopy in $5$-space.
In the next section, we provide a third and less standard proof which relies on a notion of colorings for broken surface diagrams. 

\subsection{Colorings of broken surface diagrams}

As already seen, a (singular) broken surface diagram $D$ is an immersed oriented surface in
$\R^3$, with small bands removed to indicate the different projecting
heights of the sheets. We define the \emph{regions of $D$} as the
connected components of $D$ considered with these small bands and the singular points 
removed. Locally, there are hence three regions near a regular double
point,
seven near a triple point, two near a singular and only one near a branch
point; several of these local regions can however be the same if they
are otherwise connected. Now, let $p$ be a regular double point of $D$, and
denote by $S_{\! o}$ and $S_{\! u}$ the sheets of $D$ that meet at $p$
such that $S_{\! o}$ is over $S_{\! u}$ with respect to the projection. We shall
call \emph{over-region of $p$} the region which belongs to $S_{\! o}$, and
\emph{under-regions of $p$} the other two. An 
under-region shall moreover be called \emph{positive} or
\emph{negative}, depending on whether a basis of $T_p\R^3$ made of a
positive basis for $T_pS_{\! o}$ concatenated with a vector of $T_pS_{\! u}$
which points to the considered under-region, is positive or negative,
see Figure \ref{fig:Regions} for an illustration. We also call
\emph{$i$th bottom and top regions} the unique regions that contain,
respectively, $C_i\times\{0\}$ and $C_i\times\{1\}$ on their boundary.
In the following, and
for any $a,b\in\RFn$, we
shall denote by $a^b:=b^{-1}ab$ the conjugate of $a$ by $b$.

\begin{figure}[h!]
  \[
\includegraphics{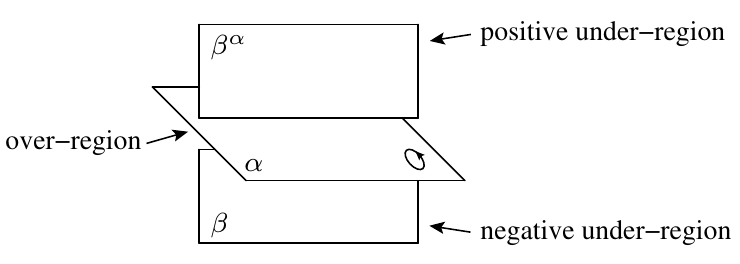}
  \]
  \caption{Regions near a regular double point}
\label{fig:Regions}
\end{figure}

\begin{defi}
  An \emph{$\RFn$--coloring} of a (singular) broken surface diagram $D$
  for a (singular) 2--string link is a labelling of the regions of $D$ by
  elements of $\RFn$ such that:
  \begin{itemize}
  \item the $i$th bottom region is
    labelled by the $i$th generator of $\RFn$;
  \item near a regular double point $p$, we have 
    $\lambda_p^+=(\lambda_p^-)^{\lambda_p^0}$ where $\lambda_p^0$, $\lambda_p^+$ and
    $\lambda_p^-$ are, respectively, the labels of the over, the positive under and the
    negative under-regions of $p$.
  \end{itemize}
\end{defi}
In particular, there is no condition assigned to singular, triple
nor branch points.
\begin{fundex}\label{ex:Wirtinger}
The construction of the invariants given in the previous section
contains, as a by-product, the fact that the reduced fundamental group
of a 2--string link is $\RFn$. On the other hand, a Wirtinger
presentation of the fundamental group can be given from a broken
surface diagram, see {\it e.g.} \cite[\S~3.2.2]{CKS}, showing that the three
meridians near a regular double point satisfy the very same relation
as for $\RFn$--colorings. It follows that labelling every region by
their corresponding reduced meridians seen in $\RFn$ defines an 
$\RFn$--coloring, that we shall call the \emph{Wirtinger coloring}. The
conjugating automorphism associated to the 2--string link can, in
particular, be deduced from the Wirtinger coloring, since it sends the
$i$th generator to the label of the $i$th top region.
\end{fundex}

An $\RFn$--coloring is merely an example of surface diagram coloring, as considered for example in \cite[\S~4.1.3]{CKS}. 
As explained there (see also \cite{Rosicki1998}) $\RFn$--colorings are preserved by Roseman moves (a)--(g), in the sense that the data of the labels 
on the boundary of a 3--ball supporting such a move is sufficient to
recover in a unique and consistent way the whole labelling inside the 3--ball; this is also clear for 
 singular Roseman self-moves. 
We thus have the following: 
\begin{lemme}\label{lem:UniqueColor}
  The number of possible $\RFn$--colorings for a (singular) broken
  surface diagram of a (singular) 2--string link is invariant under Roseman (and singular Roseman) moves.
\end{lemme}
It turns out that this number is always one for embedded 2--string links: 
\begin{prop}\label{prop:WirtingerForEver}
 The Wirtinger coloring is the unique $\RFn$--coloring for any broken surface diagram of a 2--string link.
\end{prop}
\begin{proof}
It follows from Theorem \ref{thm:1} and Proposition \ref{prop:singularRosemanMoves} that any 2--string link has the same
number of $\RFn$--colorings as a ribbon one. 
By \cite[Cor.~4.34]{ABMW}, the Tube map, defined in Section \ref{sec:codimension2} below, induces a bijection between ribbon $2$-string links up to link-homotopy and Gauss diagrams up to self-arrow moves (see  \cite[Def.~4.1 and 4.8]{ABMW}). 
It is a consequence of the definitions that the colorings of a Gauss diagram, as
defined in \cite[\S~4.4.1]{ABMW}, are in one-to-one correspondence 
with the $\RFn$--colorings of the corresponding broken surface diagram. 
Since Gauss diagrams admit a unique coloring by \cite[Lem.4.20]{ABMW}, we obtain the result. 
\end{proof}

\subsection{Classification}
As noticed in the previous section, a (singular) Roseman (self-)move modifies a given
$\RFn$--coloring only inside the ball which supports the move. Labels
of the top regions, in particular, are not modified. By unicity of the
$\RFn$--coloring, claimed in Proposition \ref{prop:WirtingerForEver}, and according to the end of Example
\ref{ex:Wirtinger}, we obtain  the following as a corollary of Proposition \ref{prop:singularRosemanMoves}:
\begin{prop}\label{prop:lh}
Two link-homotopic $2$--string links induce the same conjugating
automorphism, and hence have same Milnor invariants $\mu^{(4)}_I$ for any non-repeating sequence $I$.
\end{prop}

Consequently, the map $\varphi$ factors through the quotient of $\textrm{SL}^{2}_n$ up to link-homotopy. 
Now, this same map was shown in \cite{ABMW} to classify ribbon $2$--string links up to link-homotopy. 
An immediate consequence of Theorem \ref{thm:1} and \cite[Thm.~2.34]{ABMW} is thus the following:
\begin{prop}\label{prop:aut}
  The map $\varphi$ induces a group isomorphism between link-homotopy classes of $2$--string links and $\rm{Aut}_C(RF_n)$. 
\end{prop}

And as a direct corollary, we obtain:
\begin{theo}\label{thm:2}
Milnor $\mu^{(4)}$--invariants classify $2$--string links up to link-homotopy.
\end{theo}

\begin{remarque}\label{rem:milnor}
There are $\sum_{k=2}^n \binom{n}{k}k!$ Milnor homotopy invariants, which is the number of sequences without repetitions of length up to $n$ and at least $2$ ; these invariants, however, are not all independent.
But, as recalled in Section \ref{sec:codimension2} below, welded string links up to self-virtualization form a group which is isomorphic to that of $2$--string links up to link-homotopy, and it is shown in \cite[Thm.~9.4]{MY} using Arrow calculus that a given subset of $\sum_{k=2}^n \frac{n!}{(n-k)!(k-1)}$ of these invariants is sufficient to classify welded string links up to self-virtualization, and that any configuration of values for these numbers can be realized. In this sense, the group of $n$--component $2$--string links up to link-homotopy has rank
$\sum_{k=2}^n \frac{n!}{(n-k)!(k-1)}$, as announced in the introduction. 
This is to be compared to the rank of the group of $n$--component ($1$-dimentional) string links up to link-homotopy, which is 
$\sum_{k=2}^n \frac{n!}{(n-k)!k(k-1)}$, see \cite[\S 3]{HL}. 
\end{remarque}

\section{Link-homotopy classification in higher dimensions}\label{sec:codimension2}

The link-homotopy classification of ribbon $2$--string links, used above, was proved in \cite{ABMW} using \emph{welded} knot theory.
Loosely speaking, an $n$--component \emph{welded string link} is a proper immersion of $n$ oriented arcs, in a square with $n$ marked points 
on the top and bottom faces, such that the $i$th arc runs from the $i$th bottom to the $i$th top point, and such that singularities are transverse double 
points which are decorated either as a classical or a virtual crossing. 
These objects form a monoid $\wSL_n$ when regarded up to the usual moves of virtual knot theory \cite{Kauffman}, 
and the additional \emph{overcrossings commute} move, which allows an arc to pass \emph{over} a virtual crossing (passing under being still forbidden). 
 
S.~Satoh proved in \cite{Satoh} that there is a surjective $\Tube$ map from welded diagrams to ribbon surfaces in $4$--space, see also \cite{IndianPaper} for an alternative approach. 
We observe here that this remains actually true in higher dimensions, and 
that the classification of ribbon $2$--string links up to link-homotopy of \cite{ABMW} generalizes in higher dimensions; 
this emphasizes the fact that ribbon objects somehow stabilize at dimension $2\looparrowright4$. 
Since this fact is certainly well-known to experts, we will only outline the construction here.

For $k> 2$, a \emph{ribbon $k$--string link} is the natural higher-dimensional analogue of ribbon $2$--string links, 
{\it i.e.} the isotopy class of an embedding of copies of $S^{k-1}\times [0,1]$ in $D^{k+2}$, with similar boundary conditions, 
and bounding immersed $(k+1)$--balls which only intersect at ribbon singularities. 
Here, a ribbon singularity between two immersed $(k+1)$--balls $B$ and $B'$ is a $k$--ball whose preimages are two copies, 
one in the interior of, say, $B$ and the other with boundary embedded in $\partial B'$.
Ribbon $k$--string links with $n$ components form a monoid, denoted by $\rPS_n^k$.

Satoh's $\Tube$ map generalizes naturally to a map 
$$\Tube_k: \wSL_n\longrightarrow \rPS_n^k $$
as follows. 
For each classical crossings of a given welded string link $D$, pick two $(k+1)$--balls which share a unique ribbon singularity and are disjoint from all the other pairs. These two immersed balls should be thought of as $(k+1)$--dimensional incarnations of the two strands involved in the crossing of $D$, the ball having the preimage of the singularity in its interior corresponding to the overstrand. 
Next, it remains to connect these immersed balls to one another and to the boundary of $D^{k+2}$ by further disjointly embedded $(k+1)$--balls, 
as combinatorially prescribed by the diagram $D$. The boundary of the resulting immersed $(k+1)$--balls is the desired ribbon $k$--string link.
The key to the fact that this assignment yields a well-defined, surjective 
map is, roughly speaking, that in dimension $\ge 4$, 
ribbon knotted objects in codimension $2$ behave like framed $1$--dimensional objects. Consequently, an element of $\rPS_n^k$ is uniquely determined by the combinatorial interconnections of its ribbon singularities. 

There is a local operation on ribbon $k$--string links, defined as the deletion of a ribbon singularity by pushing it out of the immersed $(k+1)$--ball.  
%There is a local operation on ribbon $k$--string links, which results in the deletion of a ribbon singularity. 
This is a natural analogue, for ribbon knotted objects, of the usual crossing change operation. 
%, which allows to unknot any single component. 
We call \emph{ribbon link-homotopy} the equivalence relation $\sim_h$ on ribbon $k$--string links 
generated by this local move applied on ribbon singularities which have both their preimages in the same connected component. 
Note that this allows to unknot any single component. \\
It is easily checked that the \emph{self-virtualization move} on welded string links, 
which replaces a classical crossing involving two strands of a same component by a virtual one, 
generates an equivalence relation $\sim_v$ such that the map $\Tube_k$ descends to a surjective map 
$$\Tube^h_k: \fract{\wSL_n} \sim_v\longrightarrow \fract{\rPS_n^k} \sim_h. $$

Now, there is a natural map $\varphi_k$ from $\fract{\rPS_n^k} \sim_h$ to the group $\rm{Aut}_C(RF_n)$ of conjugating automorphisms of the reduced free group, defined by a straightforward generalization of the construction given in Section \ref{sec:lhinvariants} for $k=2$: roughly speaking, this action expresses the 
``top meridians'' of a ribbon $k$--string link as conjugates of the ``bottom ones''.
Moreover, we also have a map $\varphi_w$ from $\fract{\wSL_n} \sim_v$ to $\rm{Aut}_C(RF_n)$, which is known to be an isomorphism \cite[Thm.~3.11]{ABMW}, 
and which is compatible with the previous map in the sense that 
\begin{equation}\label{eq:diagram}
 \varphi_k\circ \Tube^h_k = \varphi_w. 
\end{equation}
The point here is that the fundamental group of the exterior of a ribbon $k$--string link admits a Wirtinger-type presentation, 
with a conjugating relation given at each ribbon singularity, and that the Tube map acts faithfully on the peripheral system. 
This is shown in \cite[\S~3.3]{ABMW} for $k=2$, and remains true in higher dimensions, 
owing to the fact that there is a deformation retract of the ribbon-immersed $(k+1)$--balls bounded by an element of $\rPS_n^k$ 
to ribbon-immersed $3$--balls, bounded by an element of $\rPS_n^2$. 
Combining (\ref{eq:diagram}) with \cite[Thm.~3.11]{ABMW}, we thus obtain the following classification result: 
\begin{theo}\label{thm:autk}
  The map $\varphi_k$ induces a group isomorphism between ribbon link-homotopy classes of ribbon $k$--string links and $\rm{Aut}_C(RF_n)$. 
\end{theo}
\noindent As in Section \ref{sec:lhinvariants}, this statement can be reformulated in terms of higher-dimensional Milnor invariants without repetitions.  
\begin{remarque}
 Theorem \ref{thm:autk} can be promoted to a link-homotopy classification of ribbon $k$--tori, 
 {\it i.e.} of copies of $S^{k-1}\times S^1$ bounding ribbon $(k+1)$--dimensional solid tori. 
This is done using the natural closure operation from $k$--string links to $k$--tori, and the Habegger-Lin classification scheme of \cite{HL2}, as in  \cite[\S~2.4]{ABMW} which treats the case $k=2$. 
\end{remarque}

\begin{remarque}
It is natural to ask whether one can remove the ribbon assumption in the classification Theorem \ref{thm:autk}, 
as done for $k=2$ in the present paper. 
Recall from the introduction that Bartels--Teichner's theorem \cite[Thm.1]{BT}, which is one of the keys of our proof, holds in any dimension; 
we expect that this fact could be used to attack this question. 
\end{remarque}

\appendix

\section{Non surjectivity of the braid-closure map}
\label{app:NonSurjectivity}

The braid-closure of a $1$--component $2$--string link can be seen seen as
a knotted sphere with a $1$--handle added. 
But not any knotted torus can be obtained in this way, 
\emph{i.e.} the braid-closure map is not surjective; as a matter of fact, J. Boyle
already noticed in the last paragraph of \cite[\S~4]{Boyle} that the
``$1$--turned trefoil'' $2$--torus knot 
is not the closure of a $2$--string link. 
Up to (link-)homotopy, this $2$--torus knot is however trivial, and hence is the closure of
the trivial $2$--string link. We shall now prove that, even up to link-homotopy, the braid-closure map is not surjective.

We first define an invariant for knotted surfaces as follows. 
Let $D=D_1\sqcup\cdots\sqcup D_n$ be a broken surface diagram for an $n$--component surface--link $\mathcal L$, possibly immersed with a finite number 
of singular points for which both preimages are on the same connected component.
For each $j\in\{1,\ldots,n\}$ we denote by $T_{\!j}$ the abstract surface which lives above the $j$th component of $\mathcal L$. 
For each $i\neq j$, define $\Gamma_{\!i,j}\in H_1(T_{\!j};\Z_2)$ as the homology class
\[
\Gamma_{\!i,j}:=\sum_{\gamma\in \textrm{Dbl}^+_{i,j}}\big[\varphi_j^{-1}(\gamma)\big],
\]
where $\textrm{Dbl}^+_{i,j}$ is the set of circular loci of double points between $D_i$ and $D_j$ for which the $i$th component of $\mathcal L$ stands 
above the $j$th one, according to the projection on $D$, where 
$\varphi_j:T_{\!j}\to\R^3$ is the parametrization of the $j$th component of $\mathcal L$ composed with the projection to $D$, 
and where $[\ .\ ]$ stands for the homology class.

\begin{remarque}
  Considering $\Z_2$--coefficients for $H_1(T_{\!j})$ is enforced by the fact that elements of $\textrm{Dbl}^+_{i,j}$ are not naturally oriented.
\end{remarque}

\begin{prop}\label{prop:Invariance}
  For every $i\neq j\in\{1,\ldots,n\}$, $\Gamma_{\!i,j}$ depends only on the link-homotopy class of $\mathcal L$.
\end{prop}
\begin{proof}
  This is checked using Proposition
  \ref{prop:singularRosemanMoves}. Roseman moves (a) and (e) may
  introduce or remove a component in $\textrm{Dbl}^+_{i,j}$, but with
  a trivial homology class. The action of Roseman move (b) on
  $\Gamma_{\!i,j}$ is a band sum which does not change its homology class. 
  Roseman moves (c) and (d) and singular Roseman
  self-moves (i) and (j) preserve $\Gamma_{\!i,j}$ since they involve
  only a single connected component. Roseman move (f) may only add or remove a trivial kink in a component of $\textrm{Dbl}^+_{i,j}$. 
  Finally, Roseman move (g) and singular Roseman self-move (h) obviously preserve $\textrm{Dbl}^+_{i,j}$.
\end{proof}

\begin{remarque}
  The invariant $\Gamma_{\!i,j}$ can alternatively be defined by summing over $\textrm{Dbl}^-_{i,j}$, 
  the set of circular loci where the $i$th component is below the $j$th one. 
  Indeed, this similarly defines an invariant of link-homotopy, 
  and taking their sum corresponds to summing over all circular loci in $D_i\cap D_j$, regardless of the over/under information. 
  But one can check, using Proposition \ref{prop:singularRosemanMoves}, that this sum is invariant under general homotopy;  
  since $D_i$ and $D_j$ can be pulled appart up to homotopy, the sum vanishes.
\end{remarque}

\begin{lemme}\label{lem:Obstruction}
  If $\mathcal L$ is a $2$--torus link which is the braid-closure of a $2$--string link, then for any $j\in\{1,\ldots,n\}$, 
  the set $\big\{\Gamma_{\!i,j}\ |\ i\in\{1,\ldots,n\}\setminus\{j\}\big\}$ cannot contain two distinct non-zero elements. 
\end{lemme}
\begin{proof}
  Consider $D$ a broken surface diagram for the $2$--string link closed with trivial $1$--handles. 
  Since the singular locus of $D$ does not meet the 1--handles, every element of $\textrm{Dbl}^+_{i,j}$ is necessarily parallel 
  to a multiple (possibly null) of the top boundary $\partial_{\!j}^1$ of the $j$th component of the $2$--string link; 
  the homology classes $\Gamma_{\!i,j}$ are hence either null or equal to $\big[\varphi_j^{-1}(\partial_{\!j}^1)\big]\in H_1(T_{\!j};\Z_2)$.
\end{proof}

\begin{prop}
  The braid-closure map is not surjective.
\end{prop}
\vspace{-.1cm}
  \hspace{-.43cm}\parbox[l]{12.43cm}{\begin{proof} 
  Let $D_0$ be the broken surface diagram obtained as
  follows. Consider the $3$--component link depicted on the right and make it
  spin around a line which is disjoint from it; while spinning, make 
  component $1$ run a full turn around component $3$. 
  It is easily computed that $\Gamma_{13}=a+b$ and $\Gamma_{23}=b$,
  where $a$ is the cycle represented by $3$ in the $3$--component link, 
  and $b$ is the cycle obtained by spinning any point of $3$. It follows by Lemma \ref{lem:Obstruction} that $D_0$ describes a $2$--torus link 
  which is not the braid-closure of any $2$--string link.
  \end{proof}}
  \parbox[c]{2cm}{
  \includegraphics{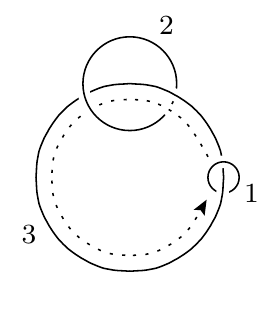}}

\begin{remarque}
  Since it is in general difficult to distinguish between the different non-trivial values for $\Gamma_{ij}$, 
  one may consider the $\Z_2$--valued invariant which only indicates whether $\Gamma_{ij}$ is zero or not. 
  For ribbon $2$--torus links, it is easily checked that this invariant coincides with the mod 2 reduction of either 
  the virtual linking number or the asymmetric linking number.
\end{remarque}

\bibliographystyle{abbrv}
\bibliography{wSL}

\begin{thebibliography}{10}

\bibitem{IndianPaper}
B.~Audoux.
\newblock {On the welded tube map.}
\newblock In {\em {Knot theory and its applications. ICTS program knot theory
  and its applications, IISER Mohali, India, 2013}}, pages 261--284.
  Providence, RI: American Mathematical Society (AMS), 2016.

\bibitem{ABMW}
B.~Audoux, P.~Bellingeri, J.-B. Meilhan, and E.~Wagner.
\newblock Homotopy classification of ribbon tubes and welded string links.
\newblock {\em {Ann. Sc. Norm. Super. Pisa, Cl. Sci. (5)}}, 17(2):713--761,
  2017.

\bibitem{BT}
A.~{Bartels} and P.~{Teichner}.
\newblock {All two dimensions links are null homotopic.}
\newblock {\em {Geom. Topol.}}, 3:235--252, 1999.

\bibitem{Boyle}
J.~{Boyle}.
\newblock {The turned torus knot in $S^4$.}
\newblock {\em {J. Knot Theory Ramifications}}, 2(3):239--249, 1993.

\bibitem{CKS}
J.~{Carter}, S.~{Kamada}, and M.~{Saito}.
\newblock {\em Surfaces in 4-space}, volume 142 of {\em Encyclopaedia of
  Mathematical Sciences}.
\newblock Springer-Verlag, Berlin, 2004.
\newblock Low-Dimensional Topology, III.

\bibitem{FQ}
M.~H. {Freedman} and F.~S. {Quinn}.
\newblock {\em {Topology of 4-manifolds.}}
\newblock Princeton, NJ: Princeton University Press, 1990.

\bibitem{HL}
N.~Habegger and X.-S. Lin.
\newblock The classification of links up to link-homotopy.
\newblock {\em J. Amer. Math. Soc.}, 3:389--419, 1990.

\bibitem{HL2}
N.~Habegger and X.-S. Lin.
\newblock On link concordance and {M}ilnor's {$\overline {\mu}$} invariants.
\newblock {\em Bull. London Math. Soc.}, 30(4):419--428, 1998.

\bibitem{Hirsch}
M.~{Hirsch}.
\newblock {Immersions of manifolds.}
\newblock {\em {Trans. Am. Math. Soc.}}, 93:242--276, 1959.

\bibitem{kamadaOJM}
S.~{Kamada}.
\newblock {Unknotting immersed surface-links and singular 2-dimensional braids
  by 1-handle surgeries.}
\newblock {\em {Osaka J. Math.}}, 36(1):33--49, 1999.

\bibitem{Kauffman}
L.~H. Kauffman.
\newblock Virtual knot theory.
\newblock {\em European J. Combin.}, 20(7):663--690, 1999.

\bibitem{Kirby}
R.~C. {Kirby}.
\newblock {\em {The topology of 4-manifolds.}}
\newblock Berlin etc.: Springer-Verlag, 1989.

\bibitem{Levine}
J.~{Levine}.
\newblock {An approach to homotopy classification of links.}
\newblock {\em {Trans. Am. Math. Soc.}}, 306(1):361--387, 1988.

\bibitem{MR}
W.~{Massey} and D.~{Rolfsen}.
\newblock {Homotopy classification of higher dimensional links.}
\newblock {\em {Indiana Univ. Math. J.}}, 34:375--391, 1985.

\bibitem{MY}
J.-B. Meilhan and A.~Yasuhara.
\newblock Arrow calculus for welded and classical links.
\newblock arXiv e-prints:1703.04658, 2017.

\bibitem{Milnor}
J.~Milnor.
\newblock Link groups.
\newblock {\em Ann. of Math. (2)}, 59:177--195, 1954.

\bibitem{Milnor2}
J.~Milnor.
\newblock Isotopy of links. {A}lgebraic geometry and topology.
\newblock In {\em A symposium in honor of {S}. {L}efschetz}, pages 280--306.
  Princeton University Press, Princeton, N. J., 1957.

\bibitem{Roseman}
D.~{Roseman}.
\newblock {Reidemeister-type moves for surfaces in four-dimensional space.}
\newblock In {\em {Knot theory. Proceedings of the mini-semester, Warsaw,
  Poland, 1995}}, pages 347--380. Warszawa: Polish Academy of Sciences,
  Institute of Mathematics, 1998.

\bibitem{Rosicki1998}
W.~Rosicki.
\newblock Some simple invariants of the position of a surface in
  $\mathbb{R}^4$.
\newblock {\em Bulletin of the Polish Academy of Sciences. Mathematics}, Vol.
  46, no 4:335--344, 1998.

\bibitem{Satoh}
S.~Satoh.
\newblock Virtual knot presentation of ribbon torus-knots.
\newblock {\em J. Knot Theory Ramifications}, 9(4):531--542, 2000.

\bibitem{stallings}
J.~Stallings.
\newblock Homology and central series of groups.
\newblock {\em J. Algebra}, 2:170--181, 1965.

\bibitem{Yaji}
T.~Yajima.
\newblock On simply knotted spheres in {$R^{4}$}.
\newblock {\em Osaka J. Math.}, 1:133--152, 1964.

\bibitem{ipipipyura}
E.~Yurasovskaya.
\newblock Homotopy string links over surfaces.
\newblock PhD Thesis, The University of British Columbia, 2008.

\end{thebibliography}

\end{document}